\documentclass{amsart}
\usepackage{amssymb}
\usepackage{mathrsfs}
\usepackage{combelow}
\usepackage{tikz}
\usepackage[linktocpage=true]{hyperref}

\newtheorem{thm}{Theorem}[section]
\newtheorem{cor}[thm]{Corollary}
\newtheorem{lem}[thm]{Lemma}
\newtheorem{prop}[thm]{Proposition}

\theoremstyle{definition}
\newtheorem{defn}[thm]{Definition}
\newtheorem{prop-def}[thm]{Proposition--Definition}

\newtheorem{nota}[thm]{Notation}

\newtheorem{exmp}[thm]{Example}

\newtheorem{Q}[thm]{Question}
\newtheorem{sit}[thm]{Situation}

\theoremstyle{remark}
\newtheorem{rmk}[thm]{Remark}

\newcommand{\zz}{{\mathbb Z}}
\newcommand{\PP}{{\mathbb P}}

\newcommand{\kk}{{\mathbb K}}
\newcommand{\jj}{{\mathrm J}}
\newcommand{\fh}{{\mathfrak h}}
\newcommand{\ord}{\operatorname{ord}}
\newcommand{\taut}{\operatorname{taut}}
\newcommand{\fami}{\mathcal}
\newcommand{\sh}{\mathscr}
\newcommand{\sE}{\sh E}
\newcommand{\sL}{\sh L}

\newcounter{ENUM}
\newcommand{\itm}{\item}
\newenvironment{ilist}[1][0]{\renewcommand{\theENUM}{\roman{ENUM}}\renewcommand{\itm}{\addtocounter{ENUM}{1}\item[(\theENUM)]}\begin{itemize}\setcounter{ENUM}{#1}}{\end{itemize}}

\usetikzlibrary{fit}

\hypersetup{colorlinks=true, citecolor=blue, linkcolor=blue, urlcolor=blue}

\begin{document}

\title{A Riemann-Hurwitz-Pl\"{u}cker formula}

\author{Brian Osserman}
\address{Department of Mathematics, UC Davis, One Shields Ave, Davis, CA 95616}
\email{osserman@math.ucdavis.edu}

\author{Adrian Zahariuc}
\address{Department of Mathematics and Statistics, University of Windsor, 401 Sunset Ave, Windsor, ON, N9B 3P4}
\email{adrian.zahariuc@uwindsor.ca}

\subjclass[2010]{Primary 14H99, 14C20.}
\keywords{Ramification point, Riemann-Hurwitz formula, Pl\"{u}cker formula}

\begin{abstract}
We prove a simultaneous generalization of the classical Riemann-Hurwitz
and Pl\"ucker formulas, addressing the total number of inflection points
of a morphism
from a (smooth, projective) curve to an arbitrary (smooth, projective) 
higher-dimensional variety. Our definition of inflection point is relative to 
an algebraic family of divisors on the target variety, and our formula is 
obtained using the theory of localized top Chern classes. In assigning 
multiplicities to inflection points, we frequently have to consider 
excess degeneracy loci, but we are able to show nonetheless that the 
multiplicities are always nonnegative, and are positive under very mild 
hypotheses.
\end{abstract}

\thanks{Brian Osserman is partially supported by grants from the Simons 
Foundation \#279151 and \#585392.}

\maketitle

\section{Introduction}\label{S: intro}

We work over an algebraically closed field ${\kk}$ of arbitrary
characteristic. 

The classical Riemann-Hurwitz and Pl\"ucker formulas
describe the total amount of inflection (or ramification) of a morphism between 
(smooth, projective) curves, and of a morphism from a curve to projective
space, respectively (see for instance \cite[Proposition 1.1]{[EH83]} for
the latter). These are both of fundamental importance to the 
study of algebraic curves, and they overlap in the case of a morphism
from a curve to the projective line, where they give the same formula.
The aim of this paper is to give a simultaneous generalization of both
formulas, to treat morphisms from a curve $C$ to an arbitrary variety $X$
(still assumed smooth and projective). We hope that this will constitute
a useful new tool in the study of curves on higher-dimensional varieties.

There does not appear to be any intrinsic
notion of inflection points of such a morphism, so our approach is to define
inflection relative to a family of divisors on $X$.
Accordingly, we work throughout in the following situation:

\begin{sit}\label{sit:basic}
Let $X$ be a smooth, projective variety, $S$ a projective variety of
dimension $n$, and
${\fami D} \subseteq X \times S$ a flat family of divisors 
on $X$.
\end{sit}

We assume that $S$ is irreducible to streamline some of the statements in
the introduction. However, we will allow $S$ to be more general in the
main body below.

\begin{defn}\label{defn:inflection intro} If we are given a smooth, projective curve $C$, and a morphism
$f:C \to X$,
we say that a $\kk$-point $p \in C$ is an \emph{inflection point} relative to ${\fami D}$ if there exists a $z \in S$ with corresponding divisor $D = {\fami D}_z \subseteq X$ such that $f^*D$ has multiplicity at least $n+1$ at $p$.\footnote{The term `ramification' is often used in the linear series context in place of inflection, but since there is a separate general notion of ramification of a morphism, we felt it would ultimately be less confusing to use the term `inflection.'} 
\end{defn}

Here we allow the possibility that $f(C) \subseteq D$ for some $D$ in
$\fami D$, in which case our convention is that every point of $C$ is 
considered an inflection point. Thus, this situation will not be 
interesting for us.

Although our setup depends on the extra data of the divisor family $\fami D$,
in practice it can often be made canonical, or dependent only on standard
data. By taking $S$ to be a connected component of the Hilbert scheme of $X$
we can reduce the extra data to an element of the Neron-Severi group.
If the Neron-Severi group of $X$ is cyclic, this 
element can be chosen canonically as the ample generator.
This is precisely what happens in the case of the Riemann-Hurwitz and
Pl\"ucker formulas. Alternatively, it is also natural to consider 
principally polarized abelian varieties, where our divisor family is 
determined by the polarization. We examine this case in \S\ref{applications}.

The formula for the total inflection of such a morphism will be expressed
in terms of two quantities $N(X,{\fami D}) \in \zz_{\geq 0}$ and 
${\sh H}(X,{\fami D}) \in \mathrm{Pic}(X)$ which arise in Situation 
\ref{sit:basic}. Specifically, for any positive integer $m$, there is a natural evaluation morphism on the space of $m$-tuples of points on divisors in ${\fami D}$, denoted
$$ \mathrm{ev}_m:\underbrace{{\fami D} \times_S {\fami D} \times_S \dots \times_S {\fami D}}_{m \text{ copies of }{\fami D}} \longrightarrow X^m. $$
This is simply the projection $X^m \times S \to X^m$ restricted to the fiber product. First, both the source and the target of $\mathrm{ev}_n$ have dimension $n\dim X$ and the target is in addition irreducible, and we define 
\begin{equation} N(X,{\fami D}) = \deg \mathrm{ev}_n. \end{equation}
Thus, $N(X,{\fami D})$ is simply the number of  divisors in $\fami D$ passing through $n$ general points in $X$.\footnote{Counted with suitable multiplicity, although the multiplicity will be $1$ if the divisors in $\fami D$ are generically reduced and the characteristic of $\kk$ is $0$.} Second, if $\smash{\alpha \in A_{(n+1)\dim X - 1}(X^{n+1}) = \mathrm{Pic}(X^{n+1}) }$ is the pushforward of the fundamental class of the source of $\smash{ \mathrm{ev}_{n+1} }$ to its target, and $\smash{ \delta:X \to X^{n+1} }$ is the small diagonal immersion, we define
\begin{equation} {\sh H}(X,{\fami D}) = \delta^*\alpha. \end{equation}
Thus ${\sh H}(X,{\fami D})$ is the pullback to $X$ under the small diagonal immersion of (roughly) the (class of the) divisor consisting of $(n+1)$-tuples of points on $X$ contained in some divisor in $\fami D$. 

Our main theorem is then the following.

\begin{thm}[Theorem \ref{thm:main theorem}]\label{main theorem} In Situation \ref{sit:basic},
if we are given any smooth proper curve $C$, and any morphism $f:C \to X$ such
that the set $I \subseteq C$ of inflection points relative to $\fami D$ is 
finite, then it is possible to assign nonnegative integer multiplicities 
$m_p$ to all 
$p \in I$ such that
\begin{equation}\label{main equation}
\sum_{p \in I} m_p p \sim f^*{\sh H}(X,\fami D) 
+ {n+1 \choose 2} N(X,\fami D) K_C,
\end{equation}  
where $\sim$ denotes linear/rational equivalence. In fact, we have $m_p > 0$ for all $p \in I$ if the induced morphism $S \to \mathrm{Hilb}(X)$ to the Hilbert scheme is generically finite, respectively $m_p = 0$ for all $p \in I$ otherwise.
\end{thm}

Note that in fact the lefthand side of \eqref{main equation} is a divisor
while the righthand side is a divisor class, so it would be more precise to say that the lefthand side is an element of the righthand side.
The theorem still holds for $S$ possibly reducible, nonreduced, or of impure dimension, and $\dim S \leq n$, with only routine adjustments to the definitions of the two quantities on the right hand side and to the statement in the last sentence, cf. Theorem \ref{thm:main theorem} below. 

Our construction is very functorial, as follows.

\begin{prop}[Proposition \ref{prop:inflection divisors are functorial}]\label{prop:functorial} In Situation \ref{sit:basic}, we
also have:
\begin{ilist}
\itm Given also $X'$ smooth and projective, 
${\fami D}' \subseteq X' \times S$ a flat family of divisors on $X'$,
and $g:X \to X'$ such that $\fami D = g^{-1} \fami D'$, then $N(X,\fami D)=N(X',\fami D')$, and ${\sh H}(X,\fami D)=g^* {\sh H}(X',\fami D')$. 
If we let $I' \subseteq C$
be the set of inflection points of $g \circ f$ relative to $\fami D'$,
and $m'_p$ for $p \in I'$ the induced multiplicities, then we also have
$I'=I$, and for each $p \in I$ we have $m'_p=m_p$.
\itm If instead we are also given $S'$ a projective variety also
of dimension $n$ and 
$h:S' \to S$, and
write $\fami D'=h^{-1} \fami D \subseteq X \times S'$, 
then
$N(X,\fami D')= \deg h \cdot N(X,\fami D)$, and 
${\sh H}(X,\fami D')= {\sh H}(X,\fami D)^{\otimes \deg h}$.
If we let $I' \subseteq C$
be the set of inflection points of $f$ relative to $\fami D'$,
and $m'_p$ for $p \in I'$ the induced multiplicities, then we also have
$I' \subseteq I$ and in fact $I'=I$ if $h$ is dominant, and for each $p \in I'$ we have $m'_p= \deg h \cdot m_p$.
\end{ilist}
\end{prop}

The corresponding statement is more complicated when $S$ is reducible,
but is given in general in Proposition \ref{prop:inflection divisors are functorial}.

We will also show:

\begin{thm}\label{thm:rhp-agree} The righthand side of
\eqref{main equation} specializes to the classical Riemann-Hurwitz and
Pl\"ucker formulas in the special cases that $X$ is a curve, or that 
$X$ is a projective space, respectively. Moreover, our general construction 
of $m_p$ recovers the usual multiplicities occurring in these formulas.
\end{thm}

Finally, our main theorem only applies when a morphism has finitely
many inflection points, so it is natural to analyze when one can
guarantee that this will be the case. In fact, it turns out to be 
possible for a particular $\fami D$ to have that \emph{every} morphism
$f:C \to X$ is everywhere ramified with respect to $\fami D$. In this
case, our formula still provides a `virtual' number of inflection
points, but this number is never realized! In
\S \ref{sec:degenerate}, we discuss these issues, and in particular we 
apply a theorem of Ein, Musta\cb{t}\u{a} and Yasuda to show that provided the 
divisors $\fami D_z$ aren't too singular for any $z \in S$, there will 
always be morphisms 
$f$ which are not everywhere ramified with respect to $\fami D$.

Our fundamental technique throughout is to study the inflection in terms 
of a section of a suitable rank-$(n+1)$ jet bundle on $C \times S$. We obtain
the righthand side of \eqref{main equation} by computing the top Chern
class of this bundle, and we use the notion of localized Chern classes
relative to a section in order to define the multiplicities. What makes
this rather delicate is that the zero set of the section in question 
is frequently positive-dimensional, so that for instance positivity
of the multiplicities is not automatic. Indeed, this happens already in the
situation of the classical Pl\"ucker formula. We are therefore forced
to make direct use of the construction of localized Chern classes given
in \cite[Lemma 3.2]{[Fu98]}, taking advantage of the usual filtration of
jet bundles to obtain some control over the geometry.

\begin{rmk}\label{rem:linear-series-case} Our construction enjoys various
functoriality properties, and in particular if $\fami D$ happens to be
a base-point free linear series on $X$, then considering the composed
morphism $C \to X \to \PP^n$, our formula will simply recover the usual
Pl\"ucker formula. From this point of view, our formula is most interesting
in cases that $\fami D$ varies non-trivially in an algebraic equivalence 
class. However, considering nonlinear subfamilies of linear series can
also lead to instructive examples; see for instance Example 
\ref{ex:dual-curve} below.
\end{rmk}

\subsection*{Conventions} All varieties (and in particular curves) are
assumed irreducible, and in particular connected. For intersection theory, we will rely heavily on Fulton's book \cite{[Fu98]}. For a finite type separated scheme $Y$ over $\kk$, we write $A_*(Y) = \bigoplus A_k(Y)$, as customary. If $W$ is a closed subscheme of $Y$, then the fundamental class of $W$ in $Y$ is $\smash{ [W] = \sum_{\alpha} m_\alpha [W_\alpha] \in A_*(Y) }$, where $\smash{ W_\mathrm{red} = \bigcup_\alpha W_\alpha }$ is the decomposition of $W$ into irreducible components and $m_\alpha$ is the length of $W$ at the generic point of $W_\alpha$. The $k$th Chern classes of vector bundles ${\sh E}$ over $Y$ are operators $A_*(Y) \to A_{*-k}(Y)$, and we write $c_k(\sh E) \cap \alpha$ for $c_k(\sh E)(\alpha)$, following \cite{[Fu98]}. Finally, $\prod_{i=1}^n c_{k_i}({\sh E}_i) \cap \alpha$ means $c_{k_1}({\sh E}_1) \cap \left( \cdots \cap \left( c_{k_n}({\sh E}_n) \cap \alpha \right) \cdots \right)$ by definition, though order is unimportant by \cite[Theorem 3.2.(b)]{[Fu98]}.

\section{Preliminaries}\label{2.0} Before getting started on the actual proof of Theorem \ref{main theorem}, we record for future use a version of the fact that pushforwards and pullbacks commute in intersection theory, granted that certain conditions hold. We haven't strived for generality -- on the contrary, this is an exercise in specializing results from Fulton's book \cite{[Fu98]}.  

\begin{prop}\label{prop:not used}
Let $X_1,X_2,X_3,X_4$ be separated finite type schemes over $\kk$ and consider a diagram
\begin{center}
\begin{tikzpicture}
\matrix [column sep  = 20mm, row sep = 7mm] {
	\node (nw) {$X_4$}; &
	\node (ne) {$X_3$};  \\
	\node (sw) {$X_2$}; &
	\node (se) {$X_1$}; \\
};
\draw[->, thin] (nw) -- (ne);
\draw[->, thin] (ne) -- (se);
\draw[->, thin] (nw) -- (sw);
\draw[->, thin] (sw) -- (se);

\node at (0,-0.4) {$\xi_{21}$};
\node at (0,0.9) {$\xi_{43}$};
\node at (-1.65,0.1) {$\xi_{42}$};
\node at (1.7,0.1) {$\xi_{31}$};

\end{tikzpicture}
\end{center}
Assume that:
\begin{enumerate} 
	\item the diagram is cartesian; 
	\item $\xi_{31}$ is proper and hence so is $\xi_{42}$; 
	\item $\xi_{21}$ and $\xi_{43}$ are lci morphisms; and 
	\item the excess normal bundle \cite[page 113]{[Fu98]} of the 
		diagram has rank $0$.
\end{enumerate}
Then $\xi_{21}^* \circ \xi_{31,*} = \xi_{42,*} \circ \xi_{43}^*$ as group homomorphisms $A_*(X_3) \to A_*(X_2)$. We refer to this by saying that \emph{pushforward and pullback commute}. 
\end{prop}

In the above, $\xi_{21}^*$ and $\xi_{43}^*$ are unrefined Gysin maps, cf. condition (3).

\begin{proof}
The general form of the excess intersection formula \cite[Proposition 6.6.(c), so see also Theorem 6.3]{[Fu98]} shows that $ \xi_{43}^*: A_*(X_3) \to A_*(X_4)$ coincides with $\xi_{21}^!: A_*(X_3) \to A_*(X_4)$. 
When invoking the excess intersection formula, we imagine a second (trivial) cartesian square on top of ours. Thus we've boiled down the claim to 
$ \xi_{21}^! \circ \xi_{31,*} = \xi_{42,*} \circ \xi_{21}^!$, 
which is \cite[Proposition 6.6.(c), so see also Theorem 6.2.(a)]{[Fu98]} if we draw another trivial cartesian square, this time under ours.
\end{proof}

Proposition \ref{prop:not used} will be used several times in the following form.

\begin{cor}\label{push pull commute}
If $X,X'$ are smooth varieties, $f:X' \to X$ is an arbitrary map, and $S$ is a proper $\kk$-scheme, then the diagram
\begin{center}
\begin{tikzpicture}
\matrix [column sep  = 20mm, row sep = 7mm] {
	\node (nw) {$X' \times S$}; &
	\node (ne) {$X \times S$};  \\
	\node (sw) {$X'$}; &
	\node (se) {$X$.}; \\
};
\draw[->, thin] (nw) -- (ne);
\draw[->, thin] (ne) -- (se);
\draw[->, thin] (nw) -- (sw);
\draw[->, thin] (sw) -- (se);

\node at (0,-0.4) {$f$};
\node at (0,0.9) {$f \times {\mathbf 1}_S$};
\node at (-2.25,0) {$\mathrm{proj}_{X'}$};
\node at (2.25,0) {$\mathrm{proj}_{X}$};

\end{tikzpicture}
\end{center}
satisfies conditions (1) -- (4) in Proposition \ref{prop:not used} and in particular pushforward and pullback commute.
\end{cor}

\begin{proof} Conditions (1) and (2) in Proposition \ref{prop:not used} are automatic. Regarding condition (3), the following easy criterion for lci-ness suffices for our purposes: if $Y$, $Z$ are smooth and separated over $S$, then any $S$-morphism $Y \to Z$ is lci. Indeed, factor $Y \to Z$ as $Y \to Y \times_S Z \to Z$ via the `graph map' and note that the former map is a regular closed immersion \cite[B.7.3]{[Fu98]}, while the latter map is smooth, so $Y \to Z$ is an lci morphism, as desired. Finally, the diagram 
\begin{center}
\begin{tikzpicture}
\matrix [column sep  = 15mm, row sep = 4mm] {
	\node (nw) {$X' \times S$}; &
	\node (ne) {$X' \times X \times S$};  &
	\node (nee) {$X \times S$}; \\
	\node (sw) {$X'$}; &
	\node (se) {$X' \times X$}; &
	\node (see) {$X$}; \\
};
\draw[->, thin] (nw) -- (ne);
\draw[->, thin] (ne) -- (nee);
\draw[->, thin] (ne) -- (se);
\draw[->, thin] (nw) -- (sw);
\draw[->, thin] (sw) -- (se);
\draw[->, thin] (se) -- (see);
\draw[->, thin] (nee) -- (see);
\end{tikzpicture}
\end{center}
illustrates that (4) also holds in this situation, cf. \cite[Proposition 6.6.(c) and diagram (**)]{[Fu98]}.
\end{proof}

In a different direction, we describe some situations in which fundamental
classes behave well. 

\begin{prop}\label{prop:fund-classes} Let $X$ and $S$ be finite-type 
$\kk$-schemes, with $X$ pure-dimensional. 

\begin{ilist}
\itm If $f:X' \to X$ is an lci morphism, and both $X$ and $X'$ are 
varieties, then 
	$$(f\times \mathbf{1}_S)^*[X \times S]=[X' \times S].$$
\itm If $D$ is an effective Cartier divisor on $X \times S$, flat over $S$, 
and $\sL$ the corresponding line bundle, then 
$$[D]=c_1(\sL)\cap [X \times S] \text{ in } A_*(X \times S).$$
\end{ilist}
\end{prop}

In the above, $f^*$ denotes Gysin pullback.
Note that both statements can fail in general if $X \times S$ is replaced
by a more general non-pure-dimensional scheme; see for instance
\cite[Example 2.6.4]{[Fu98]}.

\begin{proof} For (i), write $[S]=\sum_{\alpha} \mu_{\alpha} [S_{\alpha}]$,
where the $S_{\alpha}$ are the irreducible components of $S$. Then because
$X$ and $X'$ are varieties, we have the cycle expression
$$[X \times S]= \sum_{\alpha} \mu_{\alpha} [X \times S_{\alpha}]$$
and similarly for $X'$.
It thus suffices to see that 
	$$(f \times \mathbf{1}_S)^*[X \times S_{\alpha}] = [X' \times S_{\alpha}]
\text{ in } A_*(X' \times S)$$
for all $\alpha$. But 
	$(f \times \mathbf{1}_{S_\alpha})^* [X \times S_{\alpha}]=[X' \times S_{\alpha}]$
in $A_*(X' \times S_{\alpha})$
by \cite[Example 6.2.1]{[Fu98]}, and then we obtain the desired formula
by applying \cite[Theorem 6.2 (a) and (c)]{[Fu98]} to move from
	$(f \times \mathbf{1}_{S_{\alpha}})^*$ to $(f \times \mathbf{1}_S)^!$ and from
$A_*(X' \times S_{\alpha})$ to $A_*(X' \times S)$ respectively.

For (ii), since $D$ is assumed flat over $S$, every component of $D$ 
dominates a component of $S$, and the calculation of both sides can be done 
in a neighborhood of the generic point of that component of $S$. Thus we have
reduced to the pure-dimensional case, and can apply 
\cite[Proposition 2.6(d)]{[Fu98]}.
\end{proof}

We also have the following basic statement on Gysin maps and Chern classes.

\begin{prop}\label{prop:chern-gysin} Let $f:X' \to X$ be an lci morphism of
$\kk$-schemes, and $\sh E$ a vector bundle on $X$. Then for all $k \geq 0$
and $\alpha \in A_*(X)$, we have
$$f^*(c_m(\sh E) \cap \alpha)=c_m(f^* \sh E) \cap f^*(\alpha).$$
\end{prop}

In the above, $f^*$ denotes Gysin pullback on Chow classes, while
$f^* \sh E$ is just vector bundle pullback.

\begin{proof} This is just the generalization of 
(the $Y'=Y$ case of) \cite[Proposition 6.3]{[Fu98]} to lci morphisms,
which follows by combining it with
\cite[Theorem 3.2(d) and Proposition 6.6.(b)]{[Fu98]}.
\end{proof}

We mention a consequence which will be relevant to us in relating the
statements we have given in the introduction to those in the main body
of the paper.

\begin{cor}\label{cor:divisor-pullback} Let $X,X'$ be smooth varieties,
and $f:X' \to X$ a morphism. Then Gysin pullback of divisors agrees with
pullback of the corresponding line bundles.
\end{cor}

Note that $f$ is automatically lci by Corollary \ref{push pull commute},
so Gysin pullback makes sense.

\begin{proof} By taking differences of divisors we reduce to the case
of an effective divisor $D$ with corresponding line bundle $\sL$.
Using the case that $S$ is a point of Proposition \ref{prop:fund-classes} 
(i) and (ii), together with Proposition \ref{prop:chern-gysin}, we have 
\begin{equation*} f^*[D] = f^* (c_1(\sL)\cap [X]) = c_1(f^* \sL) \cap f^* [X] = c_1(f^* \sL) \cap [X'],
\end{equation*}
which is precisely what is asserted.
\end{proof}

We will also make extensive use of relative jet bundles, so we take the
opportunity to recall a couple of standard facts about them. Note that
formation of jet bundles does not commute with pullback in general;
however, relative jet bundles do behave well with respect to base change.

\begin{prop}\label{prop:jet-basic} Let $X$ be an $S$-scheme, and $\sh F$
a coherent sheaf on $X$, and $n \geq 0$. Then
\begin{ilist}
\itm for all morphisms $S' \to S$, if $X' = X \times_S S'$, and $f:X' \to X$
is the induced morphism, then we have
$$\jj^n_{X'/S'}(f^* \sh F)=f^* \jj^n_{X/S}(\sh F);$$
\itm for all line bundes $\sh M$ on $S$, if $\pi:X \to S$ is the structure
morphism, we have 
$$\jj^n_{X/S}(\sh F \otimes \pi^* \sh M)=\left(\jj^n_{X/S}(\sh F)\right) \otimes
\pi^* \sh M.$$
\end{ilist}
\end{prop}

\begin{proof} For (i), see \cite[Proposition 16.4.5]{[GD67]}.
(ii) is an immediate consequence of the definition of jet bundles and
the projection formula for tensoring with a locally free sheaf.
\end{proof}

\section{Intersection theoretic calculations}

The ultimate goal of this section is to recover the righthand side of 
\eqref{main equation} as the top Chern class of a natural vector bundle 
constructed from $X$, $\fami D$, $C$ and $f$.
In the first part of this section, we introduce some intersection theoretic invariants of families of divisors, which slightly generalize the quantities $N$ and ${\sh H}$ defined in the introduction. We then analyze these invariants
intrinsically on $X$, and conclude by bringing $C$ and $f$ into the 
picture and computing the relevant Chern class.

Note that from this point on, we allow $S$ to be considerably more general
than was considered in the introduction.

\begin{sit}\label{sit:actual}
Let $X$ be a smooth projective variety over $\kk$, $S$ a projective 
$\kk$-scheme,
and ${\fami D} \subset X \times S$ a flat family of divisors over $S$.
Let $(\sh L,s)$ be a line bundle and section on $X \times S$ realizing 
$\fami D$.
\end{sit} 

\begin{defn}[$\fh$-classes]\label{defn:hs} In Situation \ref{sit:actual}, for any natural number $m$, let ${\fami D}_S^m \subseteq X^m \times S$ be the $m$-fold fiber power of ${\fami D}$ over $S$.
Let $\mathrm{proj}_m:X^m \times S \to X^m$ be the projection morphism, and
$\delta_m:X \to X^m$ the small immersion diagonal.
Then let 
$$ {\fh}_m(X,{\fami D})=\delta_m^* 
\mathrm{proj}_{m,*} [{\fami D}^m_S] \in A_*(X). $$
Here, $\delta_m^*$ is the Gysin construction for regular imbeddings \cite[Example 5.2.1]{[Fu98]}.
\end{defn}

Note that Definition \ref{defn:hs} recovers $N(X,{\fami D})$ and ${\sh H}(X,{\fami D})$ from \S\ref{S: intro}.

\begin{prop}\label{prop:NH-h} In Situation \ref{sit:basic}, we have 
$\fh_{n}(X,{\fami D}) = N(X,{\fami D})$ and $\fh_{n+1}(X,{\fami D}) = {\sh H}(X,{\fami D})$. 
\end{prop}

For the first identity we are abusing notation in view of the canonical identification $A_{\dim X}(X) = \zz$, while the latter statement is immediate from
Corollary \ref{cor:divisor-pullback} applied to $\delta_m$.

We begin by expressing $[\fami D^m _S]$ as an intersection.

\begin{lem}\label{lem:diagonals}
If $\pi_\ell:X^m \to X$ is projection to the $\ell$th factor, then
\begin{equation}\label{eqn:diagonals} [{\fami D}_{S}^m] = \prod_{\ell = 1}^m c_1\left( (\pi_\ell \times {\mathbf 1}_S)^*{\sh L} \right) \cap \left[ X^m \times S \right] \in A_*(X^m \times S). \end{equation}
\end{lem}

\begin{proof}
First consider the case $m=1$. Then \eqref{eqn:diagonals} simply reads $[ {\fami D} ] = c_1({\sh L}) \cap [X \times S]$, which is Proposition
\ref{prop:fund-classes} (ii).
For general $m$, first note that we may think of $\smash{ X^{\ell-1} \times {\fami D}_S^{m-\ell+1} \subset X^\ell \times {\fami D}_S^{m-\ell} }$ as a flat family of divisors on $X$ over $\smash{ X^{\ell -1} \times {\fami D}_S^{m-\ell} }$, so, by a special case of the $m=1$ case above, we have
\begin{equation}\label{eqn:D inductive step} \left[X^{\ell-1} \times {\fami D}_S^{m-\ell+1} \right] = c_1\left( (\pi_\ell \times {\mathbf 1}_S)^*{\sh L} \right) \cap \left[ X^\ell \times {\fami D}_S^{m-\ell} \right] \end{equation}
in $A_*(X^\ell \times {\fami D}_S^{m-\ell})$, and hence in $A_*(X^m \times S)$ by the projection formula \cite[Proposition 2.5.(c)]{[Fu98]}. Then \eqref{eqn:diagonals} follows applying \eqref{eqn:D inductive step} inductively. 
\end{proof}

We now give an alternate expression for the classes ${\fh}_m(X,\fami D)$.

\begin{prop}\label{prop:h-classes} In Situation \ref{sit:actual}, for any $m$, we have 
$$ {\fh}_m(X,{\fami D}) =\mathrm{proj}_{1,*}\left( c_1({\sh L})^m \cap [X \times S] \right).$$
\end{prop}

\begin{proof} Let $ \delta_m:X \to X^m$ be the diagonal immersion. By Corollary \ref{push pull commute}, $\delta_m \times {\mathbf 1}_S$ is an lci morphism. By Lemma \ref{lem:diagonals}, we have $ (\delta_m \times {\mathbf 1}_S)^*\left[ {\fami D}_S^m \right] = c_1({\sh L})^m \cap [X \times S] $ 
in $A_*(X \times S) $
using Proposition \ref{prop:chern-gysin} 
iteratively and Proposition \ref{prop:fund-classes} (i). The proposition then follows from Corollary \ref{push pull commute}.
\end{proof}

The following trivial but convenient property can be deduced, for instance, from Proposition \ref{prop:h-classes} and the projection formula \cite[Proposition 2.5.(c)]{[Fu98]}.

\begin{cor}\label{cor:h-decompose}
In Situation \ref{sit:actual}, if $\smash{ S_\mathrm{red} = \bigcup_{\alpha \in N} S_\alpha }$ is the decomposition of $S$ into irreducible components, $\mu_\alpha$ is the length of ${\sh O}_{S,\eta_\alpha}$ where $\eta_\alpha \in S$ is the generic point of $S_\alpha$, and $\smash{ {\fami D}_\alpha = S_\alpha \times_S {\fami D} }$, then
$$ \fh_m(X,{\fami D}) = \sum_{\alpha \in N} \mu_\alpha \fh_m(X,{\fami D}_\alpha) \in A_*(X). $$
\end{cor}

We also prove the portion of Proposition \ref{prop:functorial} dealing with
the functoriality of $N$ and ${\sh H}$.

\begin{lem}\label{lem:hi-functorial} In Situation \ref{sit:actual}, we
also have:
\begin{ilist}
\itm Given also $X'$ smooth and projective, 
${\fami D}' \subseteq X' \times S$ a flat family of divisors on $X'$,
and $g:X \to X'$ such that $\fami D = g^{-1} \fami D'$, then ${\mathfrak h}_m(X,\fami D)=g^*{\mathfrak h}_m(X',\fami D')$ for all $m$.
\itm If instead we are also given $S'$ and 
$h:S' \to S$, suppose that each component $S'_{\beta}$ of $S'$ either 
dominates a component of $S$ or satisfies 
$\dim h(S'_{\beta})<\dim S'_{\beta}$. Then let
$\fami D'=h^{-1} \fami D \subseteq X \times S'$, 
let
$$ S_\mathrm{red} = \bigcup_{\alpha \in N} S_\alpha \quad \text{ and } \quad {S'}_\mathrm{red} = \bigcup_{\beta \in N'} S'_\beta $$ 
be the decompositions of $S$ and $S'$ into irreducible components, and let $\smash{ \mu'_\beta }$ the length of $S'$ at the generic point of $S'_\beta$. Define the degree of $h$ over an irreducible component $S_\alpha$ to be
$$ d_\alpha = \sum_{\beta} \mu'_\beta \deg (h:S'_\beta \to S_\alpha), $$
where the sum is taken over all $\beta$ such that $h(S'_\beta) = S_\alpha$. 
Then
\begin{equation}\label{eqn:functoriality in S}
{\mathfrak h}_m(X,\fami D')= \sum_{\alpha \in N} d_\alpha \cdot {\mathfrak h}_m(X,\fami D_\alpha)
\end{equation}
where ${\fami D}_\alpha = S_\alpha \times_S {\fami D}$. 
\end{ilist}
\end{lem}

The conditions we put on case (ii) are necessary: if for instance $h:S' \to S$ is a closed immersion, then we do not obtain information on the $\mathfrak h$-classes of $S'$ from those of $S$. This is not an issue in the context of our main theorems, as only $n$-dimensional components of $S$ will give nonzero contributions in the final formula.

\begin{proof}
For (i), let ${\sh L}'$ be the line bundle on $X' \times S$ with global section $s'$ whose vanishing locus is ${\fami D}'$. Clearly, $(g \times {\mathbf 1}_S)^* {\sh L}' = {\sh L}$. By Corollary \ref{push pull commute}, $g \times {\mathbf 1}_S$ is lci. First, we have $(g \times {\mathbf 1}_S)^*[X' \times S] = [X \times S]$ by Proposition \ref{prop:fund-classes} (i), and hence
\begin{equation}\label{eq:power-pullback}
(g \times {\mathbf 1}_S)^* \! \left( c_1({\sh L}')^m \cap [X' \times S] \right) = c_1({\sh L})^m \cap [X \times S];
\end{equation}
indeed, this follows by Proposition \ref{prop:chern-gysin} iteratively.
Let $\mathrm{proj}'_1$ be the projection $X' \times S \to X'$. 
Then $\smash{ \mathrm{proj}_{1,*} \left( c_1({\sh L})^m \cap [X \times S] \right)
= g^* \mathrm{proj}'_{1,*} \left( c_1({\sh L}')^m \cap [X' \times S] \right)  }$
by \eqref{eq:power-pullback} and Corollary \ref{push pull commute}.
The conclusion follows from Proposition \ref{prop:h-classes}.

Now we prove part (ii). We begin with the case that $S$ and $S'$ are both
varieties, and $h$ is either dominant or has $\dim h(S')<\dim S'$. Then
let $d$ be the degree of $h$ over its image, so that we have 
$h_* [S']=d[S]$ in both cases. In this situation, we have
\begin{equation*}
\begin{aligned}
({\mathbf 1}_{X^m} \times h)_* [{\fami D'}_{S'}^m] 
&= ({\mathbf 1}_{X^m} \times h)_* \left\{ \prod_{\ell =1}^m c_1\left( (\pi_\ell \times {\mathbf 1}_{S'})^* ({\mathbf 1}_X \times h)^*{\sh L} \right) \cap \left[ X^m \times S' \right] \right\} \\
&= ({\mathbf 1}_{X^m} \times h)_* \left\{ \prod_{\ell =1}^m c_1\left( ({\mathbf 1}_{X^m} \times h)^*(\pi_\ell \times {\mathbf 1}_S)^* {\sh L} \right) \cap \left[ X^m \times S' \right] \right\} \\
& =   \prod_{\ell =1}^m c_1\left( (\pi_\ell \times {\mathbf 1}_S)^* {\sh L} \right)\cap ({\mathbf 1}_{X^m} \times h)_* [X^m \times S'] \\
& =  \prod_{\ell =1}^m c_1\left( (\pi_\ell \times {\mathbf 1}_S)^* {\sh L} \right)\cap d [X^m \times S] = d [{\fami D}_{S}^m]
\end{aligned}
\end{equation*}
The first line is by \eqref{eqn:diagonals} for ${\fami D}'$, the third line is obtained by iterating \cite[Proposition 2.5.(c)]{[Fu98]}, while the fourth line is by \eqref{eqn:diagonals} once more. Pushing forward to
 $X^m$ we obtain 
$\mathrm{proj}'_{m,*} [{\fami D'}_{S'}^m] 
= d \cdot \mathrm{proj}_{m,*} [{\fami D}_{S}^m]$, 
and pulling back along the diagonal immersion yields
${\mathfrak h}_m(X,\fami D')= d \cdot {\mathfrak h}_m(X,\fami D)$.

The general statement then follows from Corollary \ref{cor:h-decompose},
using that if some $S'_{\beta}$ does not dominate any $S_{\alpha}$,
then by hypothesis we have $\dim h(S'_{\beta})<\dim S'_{\beta}$,
and we can still apply our variety case with any $\alpha$ such that
$h(S'_{\beta}) \subseteq S_{\alpha}$ to conclude that
${\mathfrak h}_m(X,\fami D'_{\beta})=0$.
\end{proof}

This concludes our general discussion of $\fh$-classes. We now reintroduce
to our running situation
all the necessary components for a Riemann-Hurwitz-Pl\"ucker formula.

\begin{sit}\label{sit:actual-2} In Situation \ref{sit:actual}, suppose
further that we have fixed $n \geq \dim S$, and we have a smooth projective
curve $C$ and a morphism $f:C \to X$.
\end{sit}

In the remainder of the section, we will compute the top Chern class of a jet bundle on
$C \times S$, recovering the righthand side of \eqref{main equation}.
We also motivate this computation by observing that the set of inflection
points can be studied via the zero set of a section of this jet bundle. We begin with an analysis of inflection in a family. The definition of inflection points is essentially identical to that in Definition \ref{defn:inflection intro}.

We reiterate the definition of inflection point, even though it is no
different than in the slightly more restrictive situation treated in the
introduction.

\begin{defn} In Situation \ref{sit:actual-2}, 
an inflection point is a $\kk$-point $r \in C$ such that there exists
$z \in S(\kk)$
with $r$ having coefficient at least $n+1$ in the divisor $\fami D_z$.
\end{defn}

Our immediate task is to clarify the structure of the set of 
inflection points 
by exhibiting it as the image of the vanishing locus of a section of a 
rank-$(n+1)$ vector bundle on $C \times S$,
and making use of the theory of localized Chern classes. This will also give 
a formula for the
constructed virtual class. Let $f_S = f \times {\mathbf 1}_S$. The key observation is the following.

\begin{prop}\label{prop:ram-family} 
Let 
$\smash{ \jj^n_{C \times S/S}f_S^*{\sh L} }$ be the relative jet bundle 
of $f_S^*{\sh L}$ over $S$, and let 
$\smash{\tilde{s}}\in \Gamma(C \times S, \smash{ \jj^n_{C \times S/S}f_S^*{\sh L}})$
be the section induced by $f_S^* s$. Then the 
(set of $\kk$-points of the) zero locus of $\smash{\tilde{s}}$ is precisely
the set of pairs $(r,z) \in C(\kk) \times S(\kk)$ such that 
${\fami D}_z$ contains $r$ with coefficient at least $n+1$.

Furthermore, we have the Chern class formula
\begin{equation}\label{relatively obvious}
c_{n+1}\!\left( \jj^n_{C \times S/S}f_S^*{\sh L} \right)=
\prod_{k=0}^{n}\left( c_1(f_S^*{\sh L}) + kc_1(\omega_{C \times S/S}) \right).
\end{equation}
\end{prop}

\begin{proof} The first statement is clear from the definitions: since we
are taking a relative jet bundle, the construction commutes with restriction
to a fiber over any $z \in S$, and then we have that $f_S^* s$ vanishes 
precisely on $f_S^{-1} \fami D_z$, so that $\smash{\tilde{s}}$ vanishes 
precisely on those $p \in C$ for which $f_S^{-1} \fami D_z$ contains an 
$n$th order thickening of $p$, as desired.

For the second, we use the well-known short exact sequence
\begin{equation}\label{jet bundle ses}
 0 \longrightarrow f_S^*{\sh L} \otimes \omega_{C \times S/S}^{\otimes k} \longrightarrow \jj^k_{C \times S/S}f_S^*{\sh L} \longrightarrow \jj^{k-1}_{C \times S/S}f_S^*{\sh L} \longrightarrow 0 
\end{equation}
which holds for each $k$.
Applying this inductively, we obtain
\begin{equation}\label{basic jet bundle formula}
\left[ \jj^n_{C \times S/S}f_S^*{\sh L} \right] = \sum_{k=0}^{n} \left[ f_S^*{\sh L} \otimes \omega_{C \times S/S}^{\otimes k} \right] \text{ in $K^\circ(C \times S)$.}
\end{equation}
Clearly, $c_1(\omega_{C \times S/S}^{\otimes k}) = kc_1(\omega_{C \times S/S})$, so the Whitney formula implies that
\begin{equation}\label{c_i explicit}
c_j\left( \bigoplus_{k=0}^{n} \omega_{C \times S/S}^{\otimes k} \right) = e_j(0,1,\ldots,n) c_1(\omega_{C \times S/S})^j, \end{equation}
where the $e_j$ denote the elementary symmetric polynomials. 
Then \eqref{relatively obvious} follows from \eqref{basic jet bundle formula},
\eqref{c_i explicit}, and the well-known 
Lemma \ref{splitting principle lemma} recalled below.
\end{proof}

\begin{lem}\label{splitting principle lemma}{\cite[page 55]{[Fu98]}}
If ${\sh L}$ is a line bundle on $Y$ and ${\sh E}$ is a vector bundle of rank $r$ on $Y$, then $c_r({\sh E} \otimes {\sh L}) = \sum_{j=0}^r c_1({\sh L})^j c_{r-j}({\sh E})$.
\end{lem}

We are now ready to complete our Chern class calculation.

\begin{cor}\label{cor:RHS} In Situation \ref{sit:actual}, with notation as
in Proposition \ref{prop:ram-family}, we have
\begin{equation}\label{eqn:RHS}
\begin{aligned} 
\mathrm{proj}_{C \times S/C,*} & \left( c_{n+1}( \jj^n_{C\times S/S} f_S^*{\sh L} ) \cap [C \times S] \right) \\
& = f^*{\fh}_{n+1}(X,\fami D) 
+ {n+1 \choose 2} K_C \cdot f^*{\fh}_n(X,\fami D) \text{ in $A_*(C)$.}
\end{aligned}
\end{equation}
\end{cor}

Note that in the statement, $f^*$ is the Gysin map induced by $f$, which
is defined because $f$ is automatically an lci morphism according to
Corollary \ref{push pull commute}. Moreover, note that although $\fh_n(X,{\fami D})$ lives in $\smash{ A_*(X) } $, only its top dimensional component contributes to $K_C \cdot f^*{\fh}_n(X,\fami D)$ (formally, $c_1({\sh T}_C^\vee) \cap f^*{\fh}_n(X,\fami D)$), and thus the last quantity is simply an integer multiple of $K_C$. 

\begin{proof}
We have $c_1(\omega_{C \times S/S}) = \mathrm{proj}^*_{C \times S/C} K_C$ and $c_1(\omega_{C \times S/S})^2 = 0$, so \eqref{relatively obvious} becomes
\begin{equation*}
c_{n+1} ( \jj^n_{C\times S/S}f^*{\sh L} )= c_1(f_S^*{\sh L})^{n+1}  + 
{n+1 \choose 2} \left( \mathrm{proj}^*_{C \times S/C} K_C\right) c_1(f_S^*{\sh L})^n 
\end{equation*}
This is an equality of operators $A_*(C \times S) \to A_{*-n-1}(C \times S)$. If we apply this to the fundamental class $[C \times S]$ and then push forward to $C$, we obtain
\begin{equation*}
\begin{aligned}
\mathrm{proj}_{C \times S/C,*} & \left( c_{n+1}( \jj^n_{C\times S/S} f_S^*{\sh L} ) \cap [C \times S] \right)= \mathrm{proj}_{C \times S/C,*} \left( c_1(f_S^*{\sh L})^{n+1} \cap [C \times S] \right) \\ &+ {n+1 \choose 2} K_C \cdot \mathrm{proj}_{C \times S/C,*} \left( c_1(f_S^*{\sh L})^n \cap [C \times S] \right)
\end{aligned} 
\end{equation*}
by the projection formula \cite[Proposition 2.5.(c)]{[Fu98]}. On the other hand, for any positive integer $m$,
\begin{equation*}
\begin{aligned}
\mathrm{proj}_{C \times S/C,*} \left( c_1(f_S^*{\sh L})^m \cap [C \times S] \right) &= \mathrm{proj}_{C \times S/C,*} f_S^* \left( c_1({\sh L})^m \cap [X \times S] \right) \\
&= f^* \mathrm{proj}_{X \times S/X,*} \left( c_1({\sh L})^m \cap [X \times S] \right) \\
&=f^* {\fh}_m(X, {\fami D}),
\end{aligned}
\end{equation*}
	so the previous formula is precisely \eqref{eqn:RHS}. To clarify the last formula above, in the first line we use Proposition \ref{prop:chern-gysin}; note also that $f_S^*[X \times S] = [C \times S]$ by Proposition \ref{prop:fund-classes} (i). The second line is by Corollary \ref{push pull commute} and the third line is by Proposition \ref{prop:h-classes}.

The corollary follows.
\end{proof}

\section{Multiplicities}

In this section, we define and study the integer multiplicities 
$m_p$ assigned to the inflection points, proving the desired functoriality
and positivity properties.

In Situation \ref{sit:actual-2}, let
${\fami I} = {\fami I}(X,{\fami D};C,f)$ be the zero locus of the section
$\smash{\tilde{s}}$ defined in Proposition \ref{prop:ram-family}. The jet bundle in Proposition \ref{prop:ram-family} and the section $\smash{\tilde{s}}$ induce a `localized top Chern class' $[{\fami I}]^\mathrm{vir} \in A_*({\fami I})$. This construction is extremely well-known, e.g. \cite[\S14.1]{[Fu98]}, with the ultimately inessential caveat that the pure dimensionality assumption in \cite[\S14.1]{[Fu98]} isn't among ours. 

Specifically, we proceed as follows. The jet bundle above corresponds to a `physical bundle' $\smash{ {\fami J} = \mathrm{Spec}_{C \times S} \mathrm{Sym} (\jj^n_{C \times S/S}f^*{\sh L})^\vee \to C \times S }$. There are two natural sections $C \times S \to {\fami J}$: the zero section ${\mathbf 0}_{\fami J}$ and the section $\tilde{\mathbf s}$ corresponding to $\tilde{s}$. Then $\smash{ {\fami I} = {(C \times S)} \times_{\fami J} {(C \times S)} }$ via ${\mathbf 0}_{\fami J}$ and $\tilde{\mathbf s}$, and we define $[{\fami I}]^\mathrm{vir} = {\mathbf 0}_{\fami J}^! [C \times S]$, with $[C \times S]$ living on the copy of $C \times S$ embedded into ${\fami J}$ via $\tilde{\mathbf s}$. It is well-known that 
\begin{equation}\label{eqn: vir is local chern} \iota_* [{\fami I}]^\mathrm{vir} = c_{n+1}(\jj^n_{C \times S/S}f_S^*{\sh L}) \cap [C \times S] \end{equation}
	in $C \times S$, where $\iota$ is the imbedding into the $0$-section copy of $C \times S$. Indeed \cite[Proposition 14.1 (a)]{[Fu98]} still holds without the pure dimensionality assumption. Of course, if $S$ is pure, then we recover the localized Chern class of \cite[\S14.1]{[Fu98]}.

A pleasant feature of the situation at hand is that, just under the assumption that the number of inflection points on $C$ is finite, the potentially virtual class $[{\fami I}]^{\mathrm{vir}}$ gives an actual (non-virtual) class on $C$. The key observation is that, since the projection $C \times S \to C$ sends ${\fami I}$ to a finite set of points $I$, the subscheme ${\fami I}$ must be a disjoint union ${\fami I} = \bigsqcup_{p \in I} {\fami I}_p$
such that ${\fami I}_p$ is contained \emph{set-theoretically} in $\{p\} \times S$. Hence there exist uniquely determined $[{\fami I}_p]^\mathrm{vir} \in A_0({\fami I}_p)$ such that 
\begin{equation}\label{eqn: split vir} [{\fami I}]^\mathrm{vir} = \sum_{p \in I} [{\fami I}_p]^\mathrm{vir} \in A_0({\fami I}), \end{equation}
where the new classes are implicitly pushed forward to ${\fami I}$. We can therefore make the following definition.

\begin{defn}\label{def:mp} In Situation \ref{sit:actual}, let $f:C \to X$ be a morphism from a curve and assume that the set of inflection points is finite. Then, with notation as above, for each $p \in I$ set $\smash{ m_p = \deg [{\fami I}_p]^\mathrm{vir} \in \zz }$, the \emph{multiplicity} of $p$, and define the \emph{inflection divisor} ${\mathbf I} = {\mathbf I}(X,{\fami D};C,f)$ as
$$ {\mathbf I} = \sum_{p \in I} m_p p \in \mathrm{Div}(C). $$ 
\end{defn}

In light of \eqref{eqn: vir is local chern}, \eqref{eqn: split vir} and Definiton \ref{def:mp}, Corollary \ref{cor:RHS} gives the most general form of the purely intersection-theoretic part of our main theorem. The primary purpose of this section is to prove the remainder of our main theorem, which we state in its general form as follows.

\begin{thm}\label{thm:main theorem} In Situation 
\ref{sit:actual}, for any smooth proper curve $C$, and any morphism $f:C \to X$ such
that the set of inflection points $I=I(X,{\fami D};C,f)$ is 
finite, we have
\begin{equation}\label{main equation 2}
{\mathbf I}(X,{\fami D};C,f) \sim f^*{\fh}_{n+1}(X,\fami D) 
+ {n+1 \choose 2} K_C \cdot {\fh}_n(X,\fami D).
\end{equation}  

Moreover, $m_p \geq 0$ for all $p \in I$, and we have $m_p > 0$ if and only if there exists an irreducible component $S_0$ of $S$ whose image under the map $S \to \mathrm{Hilb}(X)$ induced by ${\fami D}$ has dimension $n$, and a closed point $z \in S_0$ such that $f^*{\fami D}_z$ has multiplicity at least $n+1$ at $p$.
\end{thm}

Of course, the left hand side of \eqref{main equation 2} is a `physical' divisor on $C$, while the right hand side is a divisor class. Note that $f^*$ is Gysin pullback, which exists since $f$ is automatically lci (Corollary \ref{push pull commute}). Note also that this recovers Theorem \ref{main theorem} as a special case, using Corollary \ref{cor:divisor-pullback} and Proposition \ref{prop:NH-h}. If $S$ has components of dimension less than $n$, then 
${\fh}_{n}(X,\fami D)$ and ${\fh}_{n+1}(X,\fami D)$ may contain 
smaller-dimensional cycles, but these will not contribute to the righthand side of \eqref{main equation 2}.

We record for completeness and future use the remark that inflection divisors behave in the obvious way relative to decomposing into irreducible components:

\begin{prop}\label{prop:inflection divisors decompose}
Let $S_\mathrm{red} = \bigcup_{\alpha \in N} S_\alpha$ be the decomposition of $S$ into irreducible components. If ${\fami D}_\alpha = S_\alpha \times_S {\fami D}$ and $\mu_\alpha$ is the length of $S$ at the generic point of $S_\alpha$, then $I(X,{\fami D}) = \bigcup_{\alpha \in N} I(X,{\fami D}_\alpha)$ and
$$ \mathbf{I}(X,{\fami D}) = \sum_{\alpha \in N} \mu_\alpha \mathbf{I}(X,{\fami D}_\alpha). $$
Moreover, if $\dim S_\alpha < n$, then $\mathbf{I}(X,{\fami D}_\alpha) = 0$. 
\end{prop}

\begin{proof}
The first part and last part are trivial. The part concerning the inflection divisors follows from the general fact that a localized top Chern class is the sum of the (pushforwards of the) localized Chern classes of the restrictions to individual components, which can be checked using \cite[Theorem 6.2. parts (a) and (c)]{[Fu98]}.
Then we get the analogous statement with ${\mathbf I}$ replaced by $[{\fami I}]^\mathrm{vir}$ and the suitable pushforwards where necessary and the claim follows by pushing forward to $C$. 
\end{proof}

We next show that inflection divisors are functorial, completing the proof of (a generalized version of) Proposition \ref{prop:functorial}.

\begin{prop}\label{prop:inflection divisors are functorial}
In Situation \ref{sit:actual-2}, we have:
\begin{ilist}
\itm Given also $X'$ smooth and projective, 
${\fami D}' \subseteq X' \times S$ a flat family of divisors on $X'$,
and $g:X \to X'$ such that $\fami D = g^{-1} \fami D'$, then 
$ I(X',\fami D';C,g\circ f)=I(X,{\fami D};C,f)$, $$\quad {\mathbf I}(X',\fami D';C,g\circ f)={\mathbf I}(X,{\fami D};C,f),$$
and 
${\mathfrak h}_m(X,\fami D)=g^*{\mathfrak h}_m(X',\fami D')$ for all $m$.
\itm If instead we are also given $S'$ with $\dim S' \leq n$
and $h:S' \to S$, and write $\fami D'=h^{-1} \fami D \subseteq X \times S'$, then $ I(X,\fami D';C,f) \subseteq I(X,{\fami D};C,f)$, and if $h$ is surjective, then equality occurs. Moreover, if $$ S_\mathrm{red} = \bigcup_{\alpha \in N} S_\alpha \quad \text{ and } \quad {S'}_\mathrm{red} = \bigcup_{\beta \in N'} S'_\beta $$ 
are the decompositions of $S$ and $S'$ into irreducible components, $\smash{ \mu'_\beta }$ is the length of $S'$ at the generic point of $S'_\beta$, and we define the degree of $h$ over an irreducible component $S_\alpha$ to be
$$ d_\alpha = \sum_{\beta} \mu'_\beta \deg (h:S'_\beta \to S_\alpha), $$
where the sum is taken over all $\beta$ such that $h(S'_\beta) = S_\alpha$, then
\begin{equation}\label{eqn:functoriality in S-2}
{\mathbf I}(X,\fami D';C,f)=\sum_{\alpha \in N} d_\alpha {\mathbf I}(X,{\fami D}_\alpha;C,f),
\end{equation}
where ${\fami D}_\alpha = S_\alpha \times_S {\fami D}$. 
If further we have that each component $S'_{\beta}$ of $S'$ either dominates 
a component of $S$ or satisfies $\dim h(S'_{\beta})<\dim S'_{\beta}$,
then
\begin{equation*}
{\mathfrak h}_m(X,\fami D')= \sum_{\alpha \in N} d_\alpha \cdot {\mathfrak h}_m(X,\fami D_\alpha).
\end{equation*}
\end{ilist}
\end{prop}

Note that we may ignore components of $S$ of dimension less than $n$ in \eqref{eqn:functoriality in S-2}. Note too that (in light of Proposition \ref{prop:NH-h} and Corollary \ref{cor:divisor-pullback})
Proposition \ref{prop:inflection divisors are functorial} subsumes Proposition \ref{prop:functorial}, as all conditions are satisfied when $S$ and $S'$ are both varieties of the same dimension. 

\begin{proof} First, the statements on $\mathfrak h$-classes are simply
restatements of Lemma \ref{lem:hi-functorial}. For inflection points and
divisors,
statement (i) is trivial, as $I$ and ${\mathbf I}$ are constructed from precisely the same 
initial line bundle and section on $C \times S$ in both cases.

For (ii), the statement that $I(X,\fami D';C,f) \subseteq I(X,{\fami D};C,f)$,
with equality when $h$ is surjective, is also trivial.
To prove \eqref{eqn:functoriality in S-2}, we observe
that $\tilde{s}'$
is obtained by pulling back $\tilde{s}$ under $\mathbf{1}_C \times h$, and
the desired statement follows from \cite[Proposition 14.1 (d)(iii)]{[Fu98]}\footnote{\cite[Proposition 14.1 (d)(iii)]{[Fu98]} clearly has an implicit assumption: for instance, if $f$ is a closed immersion, then it doesn't hold even with $\deg(X'/X) = 0$. Regardless, $f_*[X'] = \deg(X'/X)[X]$ is what's required for the proof to go through, and this condition holds whenever $\dim X \leq \dim X'$, if we define $\deg(X'/X)$ to be $0$ when $\dim X = \dim X'$ and $f$ isn't dominant, or $\dim X < \dim X'$.} in the special case when $S$ and $S'$ are both integral, and $\dim S'=n$.
We can then deduce the general case using Proposition \ref{prop:inflection divisors decompose} as follows:
\begin{equation*}
\begin{aligned}
\mathbf{I}(X,{\fami D}') &= \sum_{\beta} \mu'_\beta \mathbf{I}(X,{\fami D}'_\beta) = \sum_{\dim h(S'_\beta) = n} \mu'_\beta \mathbf{I}(X,{\fami D}'_\beta)  \\ &=\sum_{\dim h(S'_\beta) = n} \mu'_\beta \deg_{\beta,\alpha(\beta)} \mathbf{I}(X,{\fami D}_{\alpha(\beta)}) \\
	&= \sum_{\alpha:\dim S_{\alpha}=n} \sum_{h(S'_\beta) = S_\alpha} \mu'_\beta \deg h_{\beta,\alpha}\mathbf{I}(X,{\fami D}_\alpha) = \sum_{\alpha \in N} d_\alpha {\mathbf I}(X,{\fami D}_\alpha),
\end{aligned}
\end{equation*}
where $\alpha(\beta) \in N$ is defined by $h(S'_\beta) = S_{\alpha(\beta)}$ and $h_{\beta,\alpha}$ is the restriction $S'_\beta \to S_\alpha$ of $h$. 
Note that in the first line we are able to restrict to $\beta$ with
$\dim h(S'_{\beta})=n$ because if $\dim S'_{\beta}<n$ then 
$\mathbf{I}(X,{\fami D'}_{\beta})=0$ automatically, while if
$\dim S'_{\beta}=n$ but $\dim h(S'_{\beta})<n$, then we again obtain
$\mathbf{I}(X,{\fami D'}_{\beta})=0$, this time using the integral case
for any $\alpha$ with $h(S'_{\beta}) \subseteq S_{\alpha}$. 
Similarly, the final equality in the last
line arises because $\mathbf{I}(X,{\fami D}_{\alpha})=0$ when
$\dim S_{\alpha}<n$, so it doesn't matter whether or not we 
restrict to $\alpha$ with $\dim S_{\alpha}=n$.
\end{proof}

It thus remains to show the desired nonnegativity/positivity statements on
the $m_p$. Here we will use the standard filtration on the jet bundle in order to gain better control over the interaction between the global intersection theoretic definitions above and the decomposition of ${\fami I}$ into (possibly excess-dimensional) components. We develop criteria for effectivity and positivity based on the line bundles in such a filtration. This is facilitated by the alternate construction of localized top Chern classes in the presence of a complete flag filtration from \cite[Lemma 3.2]{[Fu98]}. Although we assume it is `well known' that this 
coincides with the construction above, we were not able to find an
explicit assertion to this effect in \cite{[Fu98]}, so for the sake
of completeness, we will sketch a proof at the end of this discussion (see Proposition \ref{prop:constructions-agree}). We first recall the construction in \cite[Lemma 3.2]{[Fu98]}. The general situation will be the following:

\begin{sit}\label{sit:localized-chern} Let $X$ be a scheme, $\sE$ a 
vector bundle of rank $r$ on $X$, $s \in \Gamma(X,\sE)$ a section,
and $Y \subseteq X$ the zero set of $X$. Assume further that
$\sE$ has a filtration
$$0=\sE_0 \subseteq \sE_1 \subseteq \dots \subseteq \sE_r =\sE,$$
so that $\sL_i:=\sE_i/\sE_{i-1}$ is a line bundle for $1 \leq i \leq r$.
\end{sit}

Then for any $\alpha \in A_k(X)$, Fulton constructs a class 
$\alpha' \in A_{k-r}(Y)$ such that the pushforward of $\alpha'$ to $X$
is equal to $c_r(\sE) \cap \alpha$ in $A_{k-r}(X)$ (see Lemma 3.2
and $(\ast)$ in the proof of Theorem 3.2 of \cite{[Fu98]}).
The construction proceeds inductively as follows: first, let
$\sum a_i Z_i$ be a representative of $\alpha$, and observe that $s$
induces a section $s_r$ of $\sL_r$. For each $i$, if 
$Z_i \not\subseteq Z(s_r)$, then $s_r|_{Z_i} \in \Gamma(Z_i,\sL_r|_{Z_i})$
induces a Cartier divisor, whose associated Weil divisor is a
$(k-1)$-cycle supported by construction on $Z(s_r)$. On the other hand,
if $Z_i \subseteq Z(s_r)$, then choosing any (not necessarily effective)
Cartier divisor representing $\sL_r|_{Z_i}$, and taking its associated
Weil divisor, we again obtain a $(k-1)$-cycle supported on $Z(s_r)$.
Summing over all $i$, we obtain an element of $A_{k-1}(Z(s_r))$. Now,
on $Z(s_r)$, the section $s$ induces a section 
$s_r' \in \Gamma(Z(s_r),\sE_{r-1}|_{Z(s_r)})$, and 
$Y=Z(s)=Z(s_r')\subseteq Z(s_r)$, so we replace $X$ by $Z(s_r)$, $\sE$
by $\sE_{r-1}$, and $s$ by $s_r'$, and iterate the procedure until
we obtain the desired $\alpha' \in A_{k-r}(Y)$.

In our case, we will ultimately assume that the ambient scheme is a variety of
dimension equal to the rank of $\sE$, and we will set $\alpha$ to be
the fundamental class, so that we will canonically obtain a class in
$A_0(Y)$. We will be interested in studying the degree of this class on
each connected component $Y'$ of $Y$, and an important observation is that
for this purpose, we can work up to algebraic equivalence rather than
rational equivalence. Another important observation is that we can
obtain effectivity and positivity results from the properties of
$\sL_i|_{Y'}$. To avoid unnecessary complications, we will restrict to the
case that our ambient space is a variety, eventually reducing our desired 
results to this situation. Our basic lemma is the following.

\begin{lem}\label{lem:chern-restrict} In Situation \ref{sit:localized-chern}, 
with $X$ a variety of dimension $d$, and $\alpha = [X]$, the associated
localized top Chern class $\beta \in A_{d-r}(Y)$ can be represented as 
follows: there exist
irreducible closed subsets $Y_1,\dots,Y_m$ of $Y$, positive integers
$a_1,\dots,a_m$, and subsets 
$S_1,\dots,S_m \subseteq \{1,\dots,r\}$, such that $r=|S_j|+\dim X-\dim Y_j$
for each $j$, and
$$\beta
=\sum_{j=1}^m   \iota_{j,*} \left( a_i \prod_{i \in S_j} c_1(\sL_i|_{Y_j}) \cap [Y_j] \right),$$
where $\iota_j:Y_j \hookrightarrow Y$ is the inclusion.
\end{lem}

\begin{proof} The proof is by induction on $r$, noting that after the 
$k$th step of the construction described above, what we have done is
exactly the construction for the localized top Chern class of
$\sE/\sE_{r-k}$ supported on the zero set of the section induced by $s$.
The base case $r=0$ is trivial. Now suppose we have the lemma for $r-1$,
and let $\sE'=\sE/\sE_1$, $s'$ the section of $\sE'$ induced by $s$, 
$Y'=Z(s')$, and 
$\beta'\in A_{d-(r-1)}(Y')$ the associated localized class. Let 
$s_1 \in \Gamma(Y',\sE_1|_{Y'})$ be the section induced by $s$ (note that
$\sL_1=\sE_1$). By the inductive
hypothesis, we have some $Y_1',\dots,Y'_{m'}$, $a'_1,\dots,a'_{m'}$ and
$S'_1,\dots,S'_{m'} \subseteq \{2,\dots,r\}$ such that 
$r-1=|S_j|+\dim X-\dim Y'_j$ for each $j$, and 
\begin{equation}\label{eqn:beta-induction} \beta'
=\sum_{j=1}^{m'} \iota'_{j,*} \left(a'_i \prod_{i \in S'_j} c_1(\sL_i|_{Y'_j}) \cap [Y'_j] \right). \end{equation}
By definition, we have $\beta=(\sL_1|_{Y'},Y',s_1) \cdot \beta'$, so we have to show that the operation of product with the pseudodivisor
$(\sL_1|_{Y'},Y',s_1)$ leads to an expression of the desired form, which
we carry out one $j$ at a time. 

For those $j$ with $\smash{ Y'_j \subseteq Z(s_1)=Y }$,
the product operation simply takes a divisor representative of 
$\smash{ \sL_1|_{Y'_j} }$, and in this case we get the desired expression simply by
setting $\smash{ Y_j=Y'_j }$, $\smash{ a_j=a'_j }$, and $\smash{ S_j=\{1\} \cup S'_j} $. 

On the other hand, for those $j$ with $Y'_j \not\subseteq Z(s_1)$, then
$Z(s_1)$ induces an effective Cartier divisor on $Y'_j$, say 
$b_1 Z_1+\dots+b_{\ell} Z_{\ell}$, with each $b_i>0$. Then 
\begin{equation}\label{eqn: pseudodiv-comm} (\sL|_{Y'},Y',s_1) \cdot \left( \prod_{i \in S'_j} c_1(\sL_i |_{Y'_j}) \cap [Y'_j] \right) = \prod_{i \in S'_j} c_1(\sL_i|_{Y'_j \cap Z(s_1)}) \cap \sum_{k=1}^\ell b_k[Z_k] \end{equation}
in $A_*(Y'_j \cap Z(s_1))$, by commutativity
of product with pseudodivisors \cite[Corollary 2.4.2]{[Fu98]}. We thus see that if we split
the $j$ term of \eqref{eqn:beta-induction} into $\ell$ parts, with each $Z_{j'}$ in place
of $Y'_j$, and $b_{j'} a'_j$ in place of $a'_j$ (and leaving $S'_j$ 
unchanged), we obtain an expression of the desired form. Note that, since \eqref{eqn: pseudodiv-comm} is in $A_*(Y'_j \cap Z(s_1))$, there is no difficulty dealing with the necessary pushforwards using the projection formula \cite[Proposition 2.3 (c)]{[Fu98]}. 
\end{proof}

To apply the lemma, we will need the following.

\begin{defn}\label{defn:alg-bpf} We say a line bundle $\sL$ on a variety
$X$ is \emph{algebraically base-point free} if for all $p \in X$, there
exists an effective divisor $D$ on $X$, with $p \not \in D$, and such
that $D$ is algebraically equivalent to (a representative of) $\sL$.
\end{defn}

Our general criterion is then as follows.

\begin{thm}\label{thm:positive} 
In Situation \ref{sit:localized-chern}, with $X$ a variety,
$r=\dim X$, and $\alpha=[X]$, let $\beta \in A_0(Y)$ be the associated
localized top Chern class. Suppose that $Y'$ is a connected component
of $Y$ such that $\sL_i|_{Y'}$ is algebraically base-point free for all 
$i$. Then $\deg \beta|_{Y'} \geq 0$. If in addition for every $i$ we
have that $\sL_i|_{Y'}$ has positive degree on every curve in $Y'$, then 
$\deg \beta|_{Y'} > 0$.
\end{thm}

\begin{proof} According to Lemma \ref{lem:chern-restrict}, the degree
of $\beta|_{Y'}$ is computed by (taking a positive linear combination of
classes obtained by) intersecting representatives of
restrictions of the $\sL_i$ to various irreducible closed subsets $Y_j$
of $Y'$. However, since degree is invariant under algebraic equivalence, 
we may replace these representatives by any algebraically equivalent 
classes. By hypothesis, for each $i$ there are effective classes on $Y'$ 
algebraically equivalent to representatives for $\sL_i|_{Y'}$, and not
containing any given point, so choosing these iteratively we can realize
the degree as a proper intersection, which is therefore nonnegative.

Next, if $\sL_i|_{Y'}$ has positive degree on every curve, then we see
that for each $Y_j$, every stage of the intersection on $Y_j$ must yield
a nonzero class, so we obtain the desired positivity.
\end{proof}

\begin{rmk} The algebraically base-point free condition is actually
slightly stronger than necessary for Theorem \ref{thm:positive}:
indeed, it is enough to assume that the `algebraic base locus' of
$\sL$ does not contain any curves.

Note also that the condition of having positive degree on every curve
does not imply the algebraically base-point free condition. Indeed, Mumford
constructed an example of a (smooth, projective) surface $X$ and line
bundle $\sL$ on $X$ such that $\sL$ has positive degree on every curve on
$S$, but $\sL \cdot \sL =0$; in particular, $\sL$ cannot have a 
representative which is algebraically equivalent to an effective divisor
on $X$. See \cite[Example 1.5.2]{[La04]}.
\end{rmk}

We now apply Theorem \ref{thm:positive} to our situation. 

\begin{prop}\label{prop:positive} In Situation \ref{sit:basic}, 
let $\sL$ be as in Proposition \ref{prop:ram-family}. Then for all $p \in X$, 
we have that $\sL|_{\{p\} \times S}$ is algebraically base-point free.

Moreover, if every divisor in $\fami D$ occurs over only finitely many
values of $z \in S$, then for all $p \in X$, we have that
$\sL|_{\{p\} \times S}$ has positive degree on every curve in $S$.
\end{prop}

\begin{proof} Let $s \in \Gamma(X \times S,\sL)$ also be as in
Proposition \ref{prop:ram-family}. Then by definition, for any $p' \in X$,
we have that $s|_{\{p'\} \times S}$ vanishes precisely on the set of
$z \in S$ such that $p' \in \fami D_z \subseteq X$. Given $z \in S$, choose
any $p' \in X \smallsetminus \fami D_z$, so that $z$ is not in 
$Z(s|_{\{p'\} \times S})$, which is an effective divisor representing
$\sL|_{\{p'\} \times S}$. Now, since $X$ is connected, we have
that $\sL|_{\{p\} \times S}$ and $\sL|_{\{p'\} \times S}$ are in the
same connected component of $\mathrm{Pic}(S)$, so it follows that the former
is algebraically base-point free, as desired.

Next, let $Z \subseteq S$ be a curve. Then the restriction of
$\sL|_{\{p\} \times S}$ to $Z$ has the same degree as the restriction
$\sL|_{\{p'\} \times S}$ to $Z$ for any $p' \in X$. Under our finiteness
hypothesis, we have that the divisors $\fami D_z$ must move nontrivially
as $z$ varies in $Z$, so if we choose any $z \in Z$ and 
$p' \in X \smallsetminus \fami D_z$, there are necessarily a positive,
finite number of $z' \in Z$ such that $p' \in \fami D_{z'}$,
so we have
that $Z(s|_{\{p'\} \times S})$ meets $Z$ in a positive, finite number
of points, and hence that the degree of  
$\sL|_{\{p'\} \times S}$ on $Z$ is positive, as desired.
\end{proof}

Putting everything together, we conclude the following:

\begin{cor}\label{cor:positive-2} In the situation of Theorem 
\ref{main theorem} (in particular maintaining the hypothesis that
$S$ is a variety), all the $m_p$ are nonnegative. Under the further 
hypothesis that every divisor
appears at most finitely many times in $\fami D$, then for any
inflection point $p \in C$, we have that $m_p$ is at least equal to
the number of connected components of $\fami I|_{\{p\} \times S}$, and
in particular $m_p$ is positive.
\end{cor}

\begin{proof} Iterating \eqref{jet bundle ses} induces a complete flag
of quotient bundles for our jet bundles, and taking kernels gives us
a filtration. Note that the quotient line bundles from this filtration
are precisely the same as the kernels of \eqref{jet bundle ses}, and
also that the $\smash{ \omega_{C \times S/S}^{\otimes k} }$ term
drops out upon restriction to a fiber $\{p\} \times S$.
Then applying Proposition \ref{prop:positive} and Theorem
\ref{thm:positive}, we conclude
the desired statements on the $m_p$.
\end{proof}

Morally, the above refinement of the positivity statement for the $m_p$ says 
that if more than one divisor in $\fami D$
realizes the inflectionary behavior at $p$, they each contribute to $m_p$. 
However, this is not what contributes to higher multiplicities in
either the Riemann-Hurwitz or Pl\"ucker settings: indeed, in both these
cases $\fami I|_{\{p\} \times S}$ is always connected.

\begin{rmk} Note that it is important that Theorem \ref{thm:positive}
only requires the weaker notion of being algebraically base-point free: in our
situation, if $X$ is not rationally connected, 
the line bundles $\sL|_{\{p\} \times S}$
themselves are frequently not base-point free, and indeed in interesting
cases (as in the Riemann-Hurwitz situation, or theta divisors on 
principally polarized abelian varieties) often have only one effective 
representative.

Similarly, it is important that we do not assume that $\sL$ is
algebraically base-point free on all of $C\times S$, since in our situation 
we may get negative contributions from $\omega_{C\times S/S}$ when $C$ is 
rational.
\end{rmk}

Finally, we sketch the promised argument for the equivalence of the two constructions of the localized top Chern class.

\begin{prop}\label{prop:constructions-agree} In Situation 
\ref{sit:localized-chern}, with $X$ of pure dimension $d$, the construction of Lemma 3.2 of \cite{[Fu98]} 
recalled above agrees with the construction of localized top Chern
classes in \S14.1 of \cite{[Fu98]}.
\end{prop}

\begin{proof}[Sketch of proof] We follow the idea of Remark 14.3 of
\cite{[Fu98]}, observing that the localized top Chern class is 
completely determined by two of its formal properties, and then 
verifying that the construction of Lemma 3.2 of \cite{[Fu98]} also
satisfies these properties. Specifically, we first claim that the localized
top Chern class is completely determined by the property that it 
commutes with Gysin pullback for regular imbeddings, 
and that when $s$ is a regular section,
it is equal simply to the cycle associated to the zero scheme of $s$
(i.e., properties (d)(iii) and (c) of Proposition 14.1 of \cite{[Fu98]}).
Indeed, if $E$ is the total space of $\sE$, 
and $\pi:E \to X$ the structure map, then the diagonal map 
$E \to E \times_X E$ can be viewed as a tautological section
$s_{\taut}$ of $\pi^* \sE$, 
in the sense that any section of $\sE$ is obtained by pulling
back $s_{\taut}$ under the induced morphism $X \to E$. Now,
$s_{\taut}$ is always a regular section of $\pi^* \sE$, 
and the morphism $X \to E$ induced by a section is always a regular
imbedding, proving the claim that the localized top Chern class
is uniquely determined by the two stated properties.

Thus, in order to conclude the proposition, it suffices to prove that
the construction from Lemma 3.2 of \cite{[Fu98]} satisfies the same two
formal properties. As previously remarked, the first property follows
from Lemma 3.2 and ($\ast$) in the proof of Theorem 3.2 of \cite{[Fu98]}.
The second property, in the case that $s$ is a regular section,
must be classical, but is in any case easily deduced from results in
\cite{[Fu98]}: on some neighborhood of
any generic point of $Y$ we can split our filtration of $\sE$ and trivialize 
the quotient line bundles. We then get that the line bundle sections used 
in the construction (initially defined only on successively smaller
closed sets) extend to our neighborhood, and define a regular sequence
at the generic point in question.
The desired statement then follows from Example 6.5.1(b) and Proposition
7.1(b) of \cite{[Fu98]}.
\end{proof}

We now complete the proof of our main theorem.

\begin{proof}[Proof of Theorem \ref{thm:main theorem}] As mentioned 
initially, the formula for the inflection divisor follows from
\eqref{eqn: vir is local chern}, \eqref{eqn: split vir}, Definiton \ref{def:mp}, and Corollary \ref{cor:RHS}.

On the other hand, the statements on the $m_p$ follow 
from Corollary \ref{cor:positive-2}, Proposition \ref{prop:inflection divisors decompose} and Proposition \ref{prop:inflection divisors are functorial} (ii), which together allow us to restrict $S$ to its irreducible components, and then consider their images under the map $S \to \mathrm{Hilb}(X)$ induced by ${\fami D}$.
\end{proof}

\section{Special cases}\label{applications}

In this section, we prove Theorem \ref{thm:rhp-agree}, showing that our 
formula reduces to the classical Riemann-Hurwitz and Pl\"ucker formulas in 
the appropriate special cases. The most substantive issue is verifying 
that our multiplicities agree with the classical definitions of the 
multiplicities. As a new class of examples, we also consider the case of a 
principally polarized abelian variety, with the family of divisors induced 
by the polarization. 

\subsection*{Recovering the Riemann-Hurwitz formula} Let $X$ be a smooth 
curve, $S=X$, and $\fami D$ the diagonal in $X \times S=X \times X$; i.e. 
the family of divisors consists of all single points on $X$. Clearly, 
$N(X,\fami D) = 1$. Moreover, 
${\sh H}(X,\fami D) = -K_X$ because 
the divisor in $X \times X$ of pairs of points contained in a divisor of
$\fami D$ (formally, the pushforward to $X \times X$ of the small diagonal
in $X \times X \times X$) is precisely the diagonal $\Delta$, and the
pullback of ${\sh O}_{X \times X}(\Delta)$ under the diagonal immersion 
$X \to X \times X$ is the tangent bundle ${\sh T}_X$ of $X$. 
Then Theorem \ref{main theorem} reads
\begin{equation}\label{RH wannabe}
\sum_{p \in R} m_p p \sim -f^*K_X + K_C,
\end{equation}
which is the (algebraic form of the) Riemann-Hurwitz formula, provided that the multiplicities have the same meaning as in the Riemann-Hurwitz formula. In particular, $\smash{ \sum_{p \in I} m_p = 2g-2-d(2h-2)}$. 

To prove that that the multiplicities agree, we first review how they appear in Riemann-Hurwitz. The differential 
$$ \mathrm{d}f:{\sh T}_C \longrightarrow f^*{\sh T}_X $$
is a global section of ${\sh H}\!\!om({\sh T}_C, f^*{\sh T}_X) = f^*\omega_X^\vee \otimes \omega_C$ and its vanishing locus is the usual ramification divisor
(one then further relates this divisor to the ramification indices of $f$,
but this is not relevant to us at the moment). 
Now, in our setup
we have $f_X^* {\sh O}_{X \times X}(\Delta) = {\sh O}_{C \times X}(\Gamma_f)$, 
where $\Gamma_f \subseteq C \times X$ denotes the graph of $f$. Thus our construction takes the vanishing locus of the induced section $\smash{\tilde{s}}$ of $\smash{ \jj_{{C\times X}/X}^1{\sh O}_{C \times X}(\Gamma_f) }$, specifically the section induced by the unique (up to scalars) global section of $\smash{ {\sh O}_{C \times X}(\Gamma_f) }$ vanishing precisely on $\Gamma_f$.

To compare the two notions of multiplicity, we simply express both sections
in local coordinates. Let $p \in C$ be a ramification (inflection) point, and 
$q=f(p) \in X$. Let $u$ and $v$ be local coordinates on $C$ at $p$ and on
$X$ at $q$, respectively. Then $\omega_C$ and $\omega_X$ are trivialized
locally by $\mathrm{d}u$ and $\mathrm{d}v$, and if we write 
$f^{\sharp} v=g(u) \in \kk[[u]]$, the map $f^* \omega_X \to \omega_C$
from the usual Riemann-Hurwitz formula is expressed locally by $g'(u)$,
so that the corresponding multiplicity is simply $\ord_p g'(u)$.
On the other hand, on $C \times X$ the section $s$ of the sheaf 
$\sh O_{C \times X}(\Gamma_f)$ which vanishes precisely on $\Gamma_f$ has
a local expression $v-g(u)$ in a suitable local trivialization, and the
induced section $\smash{\tilde{s}}$ of 
$\smash{ \jj_{{C\times X}/X}^1{\sh O}_{C \times X}(\Gamma_f) }$
is given in terms of the usual trivialization simply by $(v-g(u),-g'(u))$.
Taking the zero set of the first term restricts to $\Gamma_f$, while the
second term then vanishes to order $\ord_p g'(u)$ at $p$, showing the 
desired agreement.

\subsection*{Recovering the Pl\"{u}cker formula} Let $X={\mathbb P}^r$,
$S=\left({\mathbb P}^r\right)^*$, and $\fami D$ the family of all hyperplanes
in ${\mathbb P}^r$, so $n=r$. Then 
$N({\mathbb P}^r,\fami D) = 1$ and 
${\sh H}({\mathbb P}^r,\fami D) = {\sh O}_{{\mathbb P}^r}(r+1)$. The former is trivial -- there is one hyperplane through 
$r$ general points in ${\mathbb P}^r$.
For the latter, the locus of co(hyper)planar $(r+1)$-tuples of points in ${\mathbb P}^r$ is precisely the vanishing locus of the determinant, i.e. $\{p_1 \wedge p_2 \wedge \dots \wedge p_{r+1} = 0\} \subseteq ({\mathbb P}^r)^{r+1}$, whose associated line bundle ${\sh O}_{{\mathbb P}^r}(1)^{\boxtimes r+1}$ pulls back to ${\sh O}_{{\mathbb P}^r}(r+1)$ under the diagonal immersion. Then \eqref{main equation} reads
$$ \sum_{p \in I} m_p p \sim (r+1) f^* {\sh O}_{{\mathbb P}^r}(r+1) + {r+1 \choose 2} K_C. $$
Taking degrees, we obtain $\sum_{p \in I} m_p = (r+1)d + {r+1 \choose 2}(2g-2)$, which is precisely the classical Pl\"ucker formula \cite[Proposition 1.1]{[EH83]}, provided that the multiplicities have the same meaning.

Now, note that if we set $\sL={\sh O}_{{\mathbb P}^r \times S}(\fami D)$ as
usual, and $\sL':=f^*{\sh O}_{{\mathbb P}^r}(1)$, then we have $f^*_{S}{\sh L} = {\sh L'} \boxtimes {\sh O}_{S}(1)$, 
and hence
$$ \jj^r_{{C\times S} / S} f^*_{S}{\sh L}= \jj^r_{{C\times S} / S}  \left( {\sh L'} \boxtimes {\sh O}_{S}(1) \right) = \left( \jj^r_C {\sh L'} \right) \boxtimes {\sh O}_{S}(1)$$
by Proposition \ref{prop:jet-basic}.

Setting 
${\sh E} = \jj^r_C {\sh L'}$ and ${\mathbb P}V = S=({\mathbb P}^r)^*$,
the desired agreement of the definition of multiplicities is then immediate
from the below lemma, together with the usual argument for the Pl\"ucker
formula \cite[Proposition 1.1]{[EH83]}, which expresses the inflection
divisor as the vanishing divisor of a section of a line bundle (specifically,
the determinant line bundle of $\jj^r_C {\sh L'}$).

\begin{lem}
Let ${\sh E}$ be a locally free sheaf on $C$ of rank $r+1$, and $V$ an
$(r+1)$-dimensional $\kk$-vector space. Then pushing forward to $C$
induces an isomorphism
\begin{equation*} \Gamma(C \times {\mathbb P}V, {\sh E} \boxtimes {\sh O}_{{\mathbb P}V}(1)) \cong \mathrm{Hom}_{{\sh O}_C}(V \otimes {\sh O}_C, {\sh E}).
\end{equation*}
Moreover, suppose we have a section 
$s \in \Gamma(C \times {\mathbb P}V, {\sh E} \boxtimes {\sh O}_{{\mathbb P}V}(1))$ such that the zero locus of $s$ is supported over finitely many points
of $C$, let
$\bar{s} \in \mathrm{Hom}_{{\sh O}_C}(V \otimes {\sh O}_C, {\sh E})$ be the
image of $s$ under the above isomorphism, and let $p \in C$ be any point. 
Then $\ord_p \det \bar{s}$ (as a section of the line bundle
$\det \sh E$) is equal to the multiplicity of the part of the localized top 
Chern class of $s$ which is supported over $p$.
\end{lem}

\begin{proof}
We write $\mathrm{proj}_C$ and $\mathrm{proj}_{{\mathbb P}V}$ respectively for the projections to the two factors. Using the projection formula, we have
\begin{equation*}
\begin{aligned}
\mathrm{proj}_{C,*} \left({\sh E} \boxtimes {\sh O}_{{\mathbb P}V}(1)\right) &= \mathrm{proj}_{C,*} \left( \mathrm{proj}^*_C {\sh E} \otimes \left( {\sh O}_C \boxtimes {\sh O}_{{\mathbb P}V}(1) \right) \right) \\
&= {\sh E} \otimes \mathrm{proj}_{C,*} \left({\sh O}_C \boxtimes {\sh O}_{{\mathbb P}V}(1)\right) \\
&= {\sh E} \otimes (V^\vee \otimes {\sh O}_C) = {\sh H}\!\!om_{{\sh O}_C} (V \otimes {\sh O}_C,{\sh E}).
\end{aligned}
\end{equation*}
Since ${\sh H}\!\!om$ requires no sheafification, 
by taking global sections we obtain the isomorphism 
asserted in the first statement of the lemma.

For the second statement, it suffices to work locally on $C$ (use 
Proposition 2.3(c) as in Example 2.4.2 of \cite{[Fu98]}), so we
may restrict to a neighborhood $U$ of $C$ on which we have a trivialization 
of $\sh E$, and the only zero(es) of $s$ lie over $p$. Choosing coordinates
$x_i$ on ${\mathbb P}V$ we may represent $s$
as $(\sum_{i=0}^r f^0_i x_i,\dots,\sum_{i=0}^r f^r_i x_i)$, where the 
$f^j_i$ are regular functions on $U$.
 Thus, the zero set of $s$ is the intersection
of the zero sets of the $r+1$ forms $\sum_{i=0} f^j_i x_i$ for $j=0,\dots,r$.
Consequently, we observe that the zero set of $s$ is supported over 
a point $p' \in U$ if and only if $\det (f^j_i(p'))=0$. 
At the same time, the section $\bar{s}$ is represented by the matrix
$(f^j_i)$, so we obtain a set-theoretic identity which we want to 
extend to multiplicities.

We prove this by induction on $r$. The base case $r=0$ is trivial.
For $r>0$, the localized zero locus of $s$ is
defined by successive refined intersection of the $\sum_i f^j_i x_i$,
which we order from $j=r$ to $j=0$. Note that the hypothesis that the
zero set of $s$ is empty over $U \smallsetminus \{p\}$ implies that 
we cannot have $\sum_i f^r_i x_i$ vanishing identically, so it has a
naive associated effective divisor on $U \times {\mathbb P}V$, which
consists of a multiple of $\{p\} \times {\mathbb P}V$ together with a
hyperplane bundle $Z$ inside $U \times {\mathbb P}V$ (more formally, a 
section of $U \times {\mathbb P}(V^{\vee})$). Explicitly, if $t$ is a 
local coordinate on $C$ at $p$, and $e_r=\min_i \ord_p f^r_i$, then 
factoring out $t^{e_r}$ from all the $f^r_i$ induces the stated 
decomposition, with the multiplicity of $\{p\} \times \PP V$ being 
given by $e_r$. The definition of the localized zero locus of $s$ entails
computing the contributions from each component separately, and taking
the sum
For the component $e_r (\{p\} \times \PP V)$, we claim we
always get a contribution of precisely $e_r$. Indeed, each subsequent
form will be restricted to this fiber to compute the iterated intersection,
so is simply a linear form over the base field. We then see that each
iterated intersection simply yields (the class of) a linear subspace of
one smaller dimension: the zero set is either a hyperplane in the previously 
constructed cycle, in which case the new cycle is equal to this hyperplane
(with multiplicity $1$), or it contains the previously constructed cycle,
in which case the new cycle is defined by restricting the line bundle to
the old cycle and choosing a representative divisor, which again yields a 
hyperplane with multiplicity $1$. 

We then claim that the contribution
from $Z$ consists of $(\ord _p \det (f^j_i))-e_r$ points, which yields
the desired statement. First note that if we set 
$ g^j_i
=\begin{cases} t^{-e_r} f^j_i & \text{ if }j=r \\ f^j_i & \text{ if } j<r,\end{cases} $
we see that $(\ord _p \det (f^j_i))-e_r=\ord_p \det (g^j_i)$,
and by definition $Z=Z(\sum_i g^r_i x_i)$. Also note that 
because the zero set of $s$ is empty over $U \smallsetminus \{p\}$, the
zero sets of $\sum_i f^j_i x_i$ cannot contain $Z$ for any $j<r$.
It follows that the refined intersection of these zero sets
with each other and with $Z$ is equal to the refined intersection of the 
zero sets of $(\sum_i f^j_i x_i)|_Z$ for $j=0,\dots,r-1$. By hypothesis,
we have $g^r_i(p) \neq 0$ for some $i$; without loss of generality,
suppose that this holds for $i=r$. Restricting $U$ if necessary, we then
have that $g^r_r$ is invertible, and $x_0,\dots,x_{r-1}$ induce
coordinates on $Z$. In terms of these coordinates, for each $j <r$ we have 
that 
$$  \left( \sum_i f^j_i x_i \right) \Bigg|_Z
=\sum_{i=0}^{r-1} \left( f^j_i - f^j_r \frac{g^r_i}{g^r_r} \right) x_i.$$
By the induction hypothesis, the desired multiplicity is equal to 
$$\ord_p \det \left( f^j_i - f^j_r \frac{g^r_i}{g^r_r}\right)_{i,j<r}, $$
but we see that if we add a row consisting of $f^j_r$ on the bottom and
a column consisting of $0$'s in all but the last entry (and $f^r_r$ in the
last entry) on the right, the resulting matrix can be obtained by row 
operations
from $\smash{ (g^j_i) }$, so the above determinant differs from
$\smash{ \det (g^j_i) }$ by a multiple of $\smash{g^r_r}$, which doesn't affect
the order of vanishing at $p$. We thus conclude the desired statement.
\end{proof}

\subsection*{Abelian varieties} For our final class of examples, let $X=(A,\theta)$ be a principally polarized abelian variety of dimension $n$ and let ${\fami D}$ be the family of translates of $\theta$. Then $S = \mathrm{Pic}^\theta(A)$ is a homogeneous space for $A^\vee$, which is isomorphic to $A$ via the principal polarization. In this case, we will see that Theorem \ref{main theorem} says the following: if $C$ is any curve of genus $g$, and $f:C \to A$ is a morphism which is not everywhere 
inflectionary relative to ${\fami D}$, then
\begin{equation}\label{eqn: abelian varieties} \sum_{p \in I} m_p p \sim n!{n+1 \choose 2} K_C, \end{equation}
so the number of inflection points counted appropriately is $(g-1)n(n+1)!$.

Let $ \omega: A \times \mathrm{Pic}^\theta(A) \rightarrow \mathrm{Pic}^\theta(A) $
be the action map via translation. This time we compute $N$ and ${\sh H}$ from Proposition \ref{prop:h-classes}. Note that ${\fami D} = \omega^{-1}(D)$, where $D$ is the divisor which parametrizes line bundles ${\sh L} \in \mathrm{Pic}^\theta(A)$ whose (unique up to scalars) section vanishes at $0$,
and that $D$ is the image of $\theta$ under the isomorphism 
$A \to \mathrm{Pic}^{\theta}(A)$ given by $x \mapsto T_x^* {\sh O}(\theta)$.
Then the isomorphism $\sh L \mapsto \sh L(-\theta)$ from 
$\mathrm{Pic}^\theta(A)$ to $A^{\vee}$ maps $D$ to a principal polarization
of $A^{\vee}$, so
Riemann-Roch for abelian varieties \cite[\S16]{[Mu08]} says that $D^n = n!$.
By \cite[Theorem 2.5.(d)]{[Fu98]}, we have ${\fami D}^m = \omega^* (D^m)$ -- or officially
$$ c_1({\sh O}_{A \times \mathrm{Pic}^\theta(A)}({\fami D}))^m \cap \left[ A \times \mathrm{Pic}^\theta(A) \right] = \omega^*\left( c_1({\sh O}_{\mathrm{Pic}^\theta(A)}(D))^m \cap \left[ \mathrm{Pic}^\theta(A) \right] \right), $$
so ${\sh H}(A,{\fami D}) = {\sh O}_A$ and $N(A,{\fami D}) = n!$, the latter since the projection of $ A \times \mathrm{Pic}^\theta(A)$ to the first factor restricts to isomorphisms on the fibers of $\omega$. This completes the proof of \eqref{eqn: abelian varieties}.

Note that the case $n=1$ overlaps with the Riemann-Hurwitz formula, when
the target is an elliptic curve. In this case, $K_X$ is trivial, so we
see that the two formulas agree, as they must.

Moreover, in characteristic $0$, we see that our main theorem is not vacuous
for this family of examples. Indeed, provided $A$ is not a product of
lower-dimensional abelian varieties, Ein and Lazarsfeld have shown that
$\theta$ is normal \cite[Theorem 1]{[EL97]} (using also the decomposition
theorem \cite[Theorem 4.3.1]{[BL04]}),
and Koll\'ar had previously shown that $(A,\theta)$, and hence
$\theta$, is log canonical \cite[Theorem 17.13]{[Ko95]}. Thus, by 
Corollary \ref{cor:not-too-bad} below
we conclude that there exists curves $C$ and morphisms $f:C \to A$ which
have only finitely many inflection points.

\section{Discussion of degeneracy}\label{sec:degenerate}

One issue which we have not addressed is the notion of degeneracy, or
how to characterize when it will be the case that \emph{every} point of
$C$ is an inflection point relative to $\fami D$. Obviously, this will
occur in the case that $f(C)$ is wholly contained in some divisor in
$\fami D$, so we make the following definition.

\begin{defn}\label{def:map-degen} We say that the morphism $f:C \to X$ is 
\emph{degenerate} with respect to $\fami D$ if there exists $z \in S$
such that $f(C) \subseteq \fami D_z$.
\end{defn}

This is the classical notion of degeneracy which is implicitly
disallowed by the Riemann-Hurwitz and Pl\"ucker formulas, where we 
consider non-constant morphisms of curves, or morphisms to projective space
with image not contained in any hyperplane, respectively. In both of these
situations, it is part of the classical statement of the respective
formulas that (in characteristic $0$) the set of inflection points is
\emph{always} finite. This leads us to the following question:

\begin{Q}\label{ques:nondegenerate} Under what hypotheses on
$(X,\fami D)$ is it the case that whenever $f$ is nondegenerate,
we have only finitely many inflection points relative to $\fami D$?
\end{Q} 

We see that some additional hypotheses are certainly necessary. For instance,
it is in general possible to have $(X,\fami D)$ such that \emph{every}
$f:C \to X$ is everywhere ramified.

\begin{exmp}\label{ex:totally-degenerate} Suppose that $X$ is a curve, and
$\fami D=\{2q: q \in X\}$, with $S=X$, so that $n=1$. In this case,
every morphism $f:C \to X$ will be everywhere ramified.
\end{exmp} 

This leads us to the following definition.

\begin{defn} We say that the family $\fami D$ is \emph{totally degenerate}
if for all $C$ and $f:C \to X$, we have that $f$ is everywhere ramified
with respect to $\fami D$.
\end{defn}

The above behavior can be understood in terms of jet schemes as follows. 

\begin{nota}
Let $X_n$ denote the scheme of $n$-jets on $X$. 
Let $\fami D_n \subseteq X_n$ denote the
set of $n$-jets which occur as $n$-jets for some $D \in \fami D_n$.
\end{nota}

Then $X_n$ is smooth of dimension $(\dim X)(n+1)$, and $\fami D_n$ can
be understood equivalently as the image in $X_n$ of the relative $n$-jet 
scheme of $\fami D$, which is a closed subset of $X_n \times S$. 
In particular, since $S$ is assumed proper we have that $\fami D_n$ is
closed in $X_n$. We then have:

\begin{prop}\label{prop:totally-degen} The family $\fami D$ is totally
degenerate if and only if $\fami D_n = X_n$.
\end{prop}

\begin{proof} Any $f$ induces a
morphism $C_n \to X_n$, and a point $p \in C$ is an inflection point
with respect to $\fami D$
if and only if the image of $C_n|_p$ is contained inside $\fami D_n$.

Thus, if $\fami D_n =X_n$, it is clear that $\fami D_n$ is totally
degenerate. Conversely, if $\fami D_n \subsetneq X_n$, if we choose an 
$n$-jet in $X_n \smallsetminus D_n$ based at some $p \in X$ which induces 
a nonzero tangent vector at $p$, then cutting inductively by hypersurfaces
we can produce a curve in $X$ which is smooth 
at $p$ and induces the given jet; we can then let $C$ be the normalization of
this curve.
\end{proof}

If all the divisors in $\fami D$ are smooth, then $\fami D_n$ has 
dimension at most $n+(\dim X-1)(n+1)=(\dim X)(n+1)-1$, so has codimension
at least $1$ in $X_n$. In this case, we see from Proposition
\ref{prop:totally-degen} that $\fami D$ cannot be totally degenerate.
As it turns out, this generalizes considerably thanks to a criterion
of Ein, Musta\cb{t}\u{a} and Yasuda, who have shown that (in characteristic $0$)
a divisor in $X$ which is normal and has log canonical 
singularities has a jet scheme which is pure of the expected dimension 
$(\dim X-1)(n+1)$ \cite[Remark 3.4]{[EMY03]}.
We can therefore conclude the following:

\begin{cor}\label{cor:not-too-bad} If $\kk$ has characteristic $0$,
and $\fami D_z$ is normal with at worst log canonical 
singularities for all $z \in S$, then $\fami D$ is not totally degenerate.
\end{cor}

\begin{proof} Applying the Ein-Musta\cb{t}\u{a}-Yasuda theorem, the relative jet 
scheme will be 
pure of the expected dimension $n+(\dim X-1)(n+1)$, so as in the smooth
case above, we necessarily have $D_n \subsetneq X_n$ and conclude the
corollary from Proposition \ref{prop:totally-degen}.
\end{proof}

However, even when $\fami D$ is not totally degenerate, there may be 
particular morphisms which are nondegenerate but nonetheless everywhere 
ramified.

\begin{exmp}\label{ex:dual-curve} Suppose that $X=\PP^2$, and $f:C \to \PP^2$ 
is a closed immersion. Let $\fami D$ be the dual curve to $C$; i.e., the
set of lines in $\PP^2$ occurring as tangent lines to $C$. Then $n=1$,
and although $f(C)$ is not contained in any line, we do have that $f$ is
everywhere ramified relative to $\fami D$. 
\end{exmp}

Thus, we conclude that the hypotheses required to address Question 
\ref{ques:nondegenerate} must include some nontrivial conditions on
the divisors in $\fami D$ moving with sufficient flexibility.

\begin{rmk} The jet scheme approach discussed above constitutes an 
alternate, more geometric approach to everything we do. However, we have 
found that our current approach using jet bundles seems to be more
tractable overall.
\end{rmk}

We conclude with an additional question. In both the Riemann-Hurwitz
and Pl\"ucker formulas, the multiplicities $m_p$ can be expressed
explicitly in terms of suitably defined inflection/ramification indices 
(at least, in characteristic $0$). We are therefore led to ask:

\begin{Q} Can one give explicit formulas for the $m_p$, for instance
in terms of the multiplicities at $p$ that occur among the $f^* D$ for
$D \in \fami D$?
\end{Q}

\end{document}